\documentclass[12pt]{amsart}
\usepackage[utf8]{inputenc}
\usepackage[initials]{amsrefs}
\usepackage[
        colorlinks,
]{hyperref}

\usepackage{fullpage,amssymb,amsmath,dsfont}

\numberwithin{equation}{section}

\newcommand{\R}{\mathbb R}

\newcommand{\C}{\mathbb C}
\newcommand{\Z}{\mathbb Z}

\newcommand{\m}{\mathrm{meas}}

\newcommand{\M}{\mathrm{M}}

\newtheorem{X}{X}[section]
\newtheorem{cor}[X]{Corollary}
\newtheorem{lem}[X]{Lemma}

\newtheorem{thm}[X]{Theorem}

\theoremstyle{definition}

\theoremstyle{remark}

\begin{document}

\title{Chebyshev's bias without linear independence}

\author{Mounir Hayani}
\address{Institut de mathématiques de Bordeaux, Université de Bordeaux, Talence, 33400, France}
\email{mounir.hayani@math.u-bordeaux.fr}
\date{}
\maketitle

\begin{abstract}
     We confirm Chebyshev's observation that primes are strikingly more abundant in non-square residue classes modulo a fixed integer under the Generalized Riemann Hypothesis (GRH) by proving a (natural) density $1$ statement for prime counting functions in residue classes where each prime is weighted by its inverse square root. In contrast to the majority of the existing literature on the subject, we do not need to restrict to logarithmic densities to measure Chebyshev's bias, and we do not rely on any hypothesis on the zeros of $L$-functions that is stronger than GRH.
     \newline
     \indent \textit{Note:} The same type of results presented here were independently proved by Arshay Sheth~\cite{Sheth} in the general context of automorphic forms, which implies our main asymptotic. While the spirit of the proofs is similar, Sheth develops an explicit formula for the partial Euler product and uses a result due to Gallagher~\cite{Gallagher} to prove that some estimates hold outside a set of finite logarithmic measure. We share this independent work because it provides a completely self-contained and elementary proof relying only on the usual explicit formula, and it yields explicit error terms rather than an implicit $o(1)$ asymptotic.
\end{abstract}
\section{Introduction}
In a letter~\cite{Chebyshev} sent to Fuss in 1853, Chebyshev observed that primes congruent to $3\bmod\,  4$ seem to appear more often than those congruent to $1\bmod\,  4$. Let $q\ge3$ and let $a\in \Z$ be coprime to $q$, and define $\pi(x;q,a)=\#\{ p\le x\, :\,  p\equiv a\bmod\,  q\}$. Although the prime number theorem in arithmetic progressions asserts that primes are equidistributed among invertible classes modulo $q$, it gives no control over the difference $\pi(x;q,a)-\pi(x;q,b)$ when $b\in\Z$ is coprime to $q$. In 1914, Littlewood~\cite{Littlewood} proved that the function $x\mapsto\pi(x;4,1)-\pi(x;4,3)$ changes sign infinitely often. This naturally raised the question of the existence of the natural density of the set $\mathcal{P}(q;a,b):=\{x\ge 2\, :\, \pi(x;q,a)>\pi(x;q,b)\}$, namely the limit \[ \lim_{X\to \infty}\frac{\m\{2\le x\le X\, :\, x\in \mathcal{P}(q;a,b) \}}{X},\] where $\m$ is the Lebesgue measure on $\R$. This question was first addressed by Kaczorowski~\cite{Kac}, who proved under $\text{GRH}$ that the natural density of $\mathcal{P}(4;3,1)$ does not exist. 

In 1994, Rubinstein and Sarnak~\cite{RS94} established a theoretical framework that enabled them to address the question of Chebyshev's bias for general moduli. They proved, under GRH and a linear independence hypothesis on non-negative imaginary parts of the zeros of Dirichlet $L$-functions (LI), that the logarithmic density \[ \lim_{X\to \infty} \frac{1}{\log X}\int_{2}^X \mathds{1}_{\mathcal{P}(q;a,b)}(u)\frac{\mathrm{d}u}{u}\]
of the set $\mathcal{P}(q;a,b)$ exists. Furthermore, they quantified the bias for $q=4$, proving that the logarithmic density of the set $\mathcal{P}(4;3,1)$ is approximately $0.9959\dots$

More recently, a new framework for studying Chebyshev's bias was introduced by Aoki and Koyama~\cite{Ao-Ko}. They observed that since one of the reasons for the existence of the logarithmic density is that the factor $1/u$ in the integral dampens the contribution of large primes, it would be natural to consider the prime counting function \begin{equation}\label{pi_1/2}
    \pi_{1/2}(x;q,a)=\sum_{\substack{p\le x\\ p\equiv a\bmod\, {q}}}\frac{1}{\sqrt{p}}
\end{equation}
and ask whether the natural density of the set $\mathcal{P}_{1/2}(q;a,b):=\{x\ge 2\, :\, \pi_{1/2}(x;q,a)>\pi_{1/2}(x;q,b)\}$ exists. In their framework, and that of Koyama--Kurokawa~\cite{Ko-Ku}, the idea was to avoid using a hypothesis as strong as LI. Therefore, they worked in the context of the Deep Riemann Hypothesis (DRH), a hypothesis stronger than GRH introduced by K. Conrad~\cite{Conrad}, which states that for all non-principal Dirichlet characters modulo $q$ we have \begin{equation}\label{DRH} (\log x)^{m_\chi} \prod_{p\le x}\left(1-\frac{\chi(p)}{p^{1/2}}\right)^{-1}=\ell_\chi+o(1) \quad(x\to \infty) ,\end{equation}
where $\ell_\chi\ne 0$ and $m_\chi:=\text{ord}_{s=1/2}\,L(s,\chi)$ is the order of vanishing of $L(s,\chi)$ at the central point $s=1/2$. Aoki and Koyama observed that taking the logarithm in~\eqref{DRH}, and using its Taylor expansion, then using Mertens' theorem to estimate $\sum_{p\le x} \chi(p^2)/p$, leads to an estimate of $\sum_{p\le x} \chi(p)/\sqrt{p}$, which allowed them to prove that DRH holds if and only if for all invertible residue classes $a,b$ modulo $q$ one has \begin{equation}\label{DRH-estimate}
    \pi_{1/2}(x;q,a)-\pi_{1/2}(x;q,b)=-\M(q;a,b) \log \log x+C+o(1) \, ,
\end{equation}
    where $\M(q;a,b)=\frac{1}{2\varphi(q)}\left(r(a)-r(b)+2\sum_{\chi \ne \chi_0}(\chi(a)-\chi(b))m_\chi\right)$, $r:(\Z/q\Z)^\times\to \Z_{\ge 0}$ is defined by $r(a)=\#\{ x\in (\Z/q\Z)^\times\, :\, x^2=a\}$ and $C$ is a constant depending on $q,a,b$. In a very recent paper, Suzuki~\cite{Su} studied the weighted Chebyshev-type function $\psi_{1/2}(x):=\sum_{n\le x}\Lambda(n)/\sqrt{n}$ and its variants in arithmetic progressions. The author proved that GRH holds for the non-principal Dirichlet character modulo $4$ if and only if  
\[\sqrt{x}\left(\sum_{\substack{p\le x\\ p\equiv 1\bmod\, 4}}\frac{\log p}{\sqrt{p}}\log\left(\frac{x}{p}\right)-\sum_{ \substack{p\le x\\ p\equiv 3 \bmod\, 4}}\frac{\log p}{\sqrt{p}}\log \left(\frac{x}{p}\right) \right)\sim-(\log x)^2 \frac{\sqrt{x}}{4}\quad (x\to \infty)\, .\]

This suggests that studying the asymptotic behavior of the difference of prime counting functions of the same type as~\eqref{pi_1/2} requires either hypotheses stronger than GRH (such as DRH in Aoki--Koyama's work), or additional smoothing weights ($\log p \log (x/p)$ in Suzuki's work) which make the contribution of large primes smaller. 

In this paper, we show that GRH is sufficient to prove that the equality~\eqref{DRH-estimate} holds for a set of $x\ge 2$ of natural density $1$. More precisely: 
\begin{thm}\label{main}
    Let $q\ge 3$ and assume \text{GRH} for all non-principal Dirichlet characters modulo $q$. Let $\varepsilon>0$ and let $a,b$ be distinct invertible residue classes modulo $q$. Then there exists a (unique) constant $C$ depending on $q,\, a,$ and $b$, such that the natural density of the set \[\left\{x\ge 2\, :\, \left| \pi_{1/2}(x;q,a)-\pi_{1/2}(x;q,b) +\M(q;a,b)\log \log x-C\right|\le \frac{(\log \log x)^{3+\varepsilon}}{\log x}\right\}\]
    exists and equals $1$. Furthermore, there exists $K=K(a,b)>1$ such that the logarithmic density of the set 
    \[\left\{x\ge 2\, :\, \left| \pi_{1/2}(x;q,a)-\pi_{1/2}(x;q,b) +\M(q;a,b)\log \log x-C\right|\le \frac{K \log \log x}{\log x}\right\}\]
    exists and equals $1$.
\end{thm}
Assuming that $m_\chi=0$ for all $\chi \bmod\,  q$, we have $\M(q;a,b)=(r(a)-r(b))/(2\varphi(q))$. Therefore, if $r(a)<r(b)$, then assuming GRH for all non-principal Dirichlet characters modulo $q$, the natural density of the set $\mathcal{P}_{1/2}(q;a,b)$ exists and equals $1$. Thus, Theorem~\ref{main} explains Chebyshev's bias, thanks to these weighted prime counting functions, without the need for any strong additional hypotheses (such as LI or DRH). For historical reasons, we state the particular case $q=4$:
\begin{cor}
    Assume GRH for the non-principal character modulo $4$. Then, the natural density of the set $ \mathcal{P}_{1/2}(4;3,1)=\{ x\ge 2\, :\, \pi_{1/2}(x;4,3)>\pi_{1/2}(x;4,1)\}$
    exists and equals $1$.
\end{cor}
Our second main result shows that if GRH holds, then the equality~\eqref{DRH} holds for a set of natural density $1$.
\begin{thm}\label{GRHalmimpliesDRH}
    Assume GRH for a non-principal Dirichlet character $\chi$ modulo $q$. Then, there exists $\ell_\chi\ne 0$ such that, for all $\varepsilon>0$, the natural density of the set 
    \begin{equation}\label{Echi}
        \mathcal{E}_\chi:=\left\{x\ge 2\, :\, \left|  (\log x)^{m_\chi} \prod_{p\le x}\left(1-\frac{\chi(p)}{p^{1/2}}\right)^{-1}-\ell_\chi  \right|\le \frac{(\log \log x)^{3+\varepsilon}}{\log x} \right\}
    \end{equation}
    exists and equals $1$.
\end{thm}

Our last main result concerns the mean $\frac{1}{x}\int_2^x(\pi_{1/2}(u;q,a)-\pi_{1/2}(u;q,b))\mathrm{d}u$, our method allows us to obtain, under GRH, the exact estimate that one would obtain assuming the DRH estimate~\eqref{DRH-estimate}:

\begin{thm}\label{mean}
    Let $q\ge 3$ and assume \text{GRH} for all non-principal Dirichlet characters modulo $q$. Let $a,b$ be distinct invertible residue classes modulo $q$. Then 
    \[\frac{1}{x}\int_{ 2}^x \bigl(\pi_{1/2}(u;q,a)-\pi_{1/2}(u;q,b)\bigr)\mathrm{d}u=-\M(q;a,b)\log \log x+C+O\left(\frac{\log \log x}{\log x}\right)\, .\]
\end{thm}
Although our arguments extend straightforwardly to the context of global fields, we restrict our focus to the classical case of primes in arithmetic progressions for the sake of clarity and to highlight the underlying ideas.
To establish our results, we rely on the explicit formula for prime counting functions. However, a direct application of this formula to the natural auxiliary Chebyshev function $\psi_{1/2}$ associated with the weighted function $\pi_{1/2}$ would be inconclusive: the resulting error term is of order $O(\log x)$, while its main term is not expected to be that large; in fact the main term is expected to be of order $O((\log \log \log x)^2)$ by Montgomery's conjecture~\cite{Montg}. Instead, our strategy is to relate $\pi_{1/2}$ to $\pi$ via summation by parts. We show (using Lemma~\ref{ing}) that, thanks to the explicit formula for $\pi$, the integral term arising from this summation yields the main bias term
\[-\M(q;a,b) \log \log x +C + O\left(\frac{\log \log x}{\log x}\right).\]
The remaining non-integral term is essentially $(\pi(x;q,a)-\pi(x;q,b))/\sqrt{x}$. In order to bound the size of the set of $x\ge 2$ for which the deviation exceeds $(\log \log x)^{3+\varepsilon}/(\log x)$, we use Markov's inequality (Lemma~\ref{lem:nat-density}) which reduces the problem to obtaining a uniform estimate for higher moments of the error term of primes in arithmetic progressions. Building on the techniques of Puchta~\cite{Puchta}, we establish a uniform bound for these moments in Lemma~\ref{variance} by refining the estimates of the multisums over zeros, a result which may be of independent interest. Since our estimates are naturally formulated on a logarithmic scale, we deduce the existence of the natural density, by showing that the set of these deviations has finite logarithmic measure. Indeed, any Borel set $A\subset (0,\infty)$ satisfying $\int_0^\infty \mathds{1}_A(e^y)\mathrm{d}y<\infty$ has zero natural density. To see this, note that for any $\mathcal{L}>1$,
 \begin{equation}\label{nat-density}
    0\le \limsup_{x\to \infty}\frac{1}{x} \int_1^x \mathds{1}_A(u)\mathrm{d}u\le \limsup_{x\to \infty}\int_{\mathcal{L}}^x \mathds{1}_A(u) \frac{\mathrm{d}u}{u}=\int_{\log \mathcal{L}}^\infty \mathds{1}_{A}\bigl(e^y\bigr)\mathrm{d}y\, .
\end{equation}
The result follows by letting $\mathcal{L}\to \infty$.
\subsection*{Acknowledgments} 
I am very grateful to Alexandre Bailleul for his detailed reading and for kindly pointing out that similar results were independently proved by Arshay Sheth~\cite{Sheth} in the general context of automorphic forms. I also thank Sheth for subsequently reaching out to discuss this work. Finally, I would like to thank Florent Jouve and Daniel Fiorilli for their constant encouragement and for their suggestions, which were essential to the development of this paper.

\section{Preliminaries}\label{sec:Prel}
Let $q\ge3$ and let $t\, :\, (\Z/q\Z)^\times\to\C$ be a non-zero function satisfying \(\langle t,\chi_0\rangle=0\, ,\)
where $\chi_0$ is the principal character modulo $q$, and where \[ \langle f,g\rangle :=\frac{1}{\varphi(q)}\sum_{a\in (\Z/q\Z)^\times} f(a)\overline{g(a)}\]
for all functions $f,g\, :\, (\Z/q\Z)^\times\to \C$.
We define for $x\ge 2$ \[ \pi_{1/2}(x;t):=\sum_{p\le x} \frac{t(p)}{p^{1/2}}\ \text{ and }\ \pi(x;t):=\sum_{p\le x} t(p)\, \]
where the sums run over primes $p\le x$ that are invertible modulo $q$. We note that \[\pi_{1/2}(x;q,a)-\pi_{1/2}(x;q,b)=\pi_{1/2}\bigl(x; \mathds{1}_{\{a\}}-\mathds{1}_{\{b\}} \bigr).\]
Therefore, we study the case of a general function $t$ on $(\Z/q\Z)^\times$ satisfying $\langle t,\chi_0\rangle=0$.
We define the usual auxiliary prime counting functions: 
\[\psi(x;t):=\sum_{n\le x} \Lambda(n)t(n)\quad \text{and}\quad \theta(x;t):=\sum_{p\le x} t(p)\log p \, .\]
By linearity we have $\psi(x;t)=\sum_{\chi \ne\chi_0} \langle t,\chi \rangle \psi(x;\chi)$. Thus, using~\cite{MV}*{ Theorem 12.12}, and assuming GRH for all characters $\chi$ modulo $q$ satisfying $\langle t,\chi \rangle\ne 0$, the explicit formula for $\psi(x;t)$ can be written as  \begin{equation}\label{psi-expl-form}
    \psi(x;t)=-\sum_{\chi \ne \chi_0}\langle t,\chi\rangle \sum_{0\le |\gamma_\chi|\le T} \frac{x^{\frac{1}{2}+ i\gamma_\chi}}{\frac{1}{2}+i\gamma_\chi}+O\left(x\frac{(\log T)^2}{T}+\log x\right),
\end{equation}
where the implied constants depend on $q$ and $t$. We now define for all $y\ge1$: 
\begin{equation}
    \Delta(y;t):=y\frac{\pi(e^{y};t)}{\exp(y/2)}+2\M(t)\, ,
\end{equation}
where $\M(t):=\bigl(\langle t,r\rangle +2\sum_{\chi\ne \chi_0}\langle t,\chi\rangle m_\chi\bigr)/2$. Following~\cite{RS94}*{Lemma 2.1}, we deduce that, if $\text{GRH}$ holds for all Dirichlet characters modulo $q$ satisfying $\langle t,\chi\rangle \ne 0$, then uniformly for $y\ge 1$
\begin{equation}\label{expl-form}\Delta(y;t)=-\sum_{\chi \ne\chi_0}\langle t,\chi\rangle\sum_{0<|\gamma_\chi|\le T} \frac{e^{iy\gamma_\chi}}{\frac{1}{2}+i\gamma_\chi}+O\left(\frac{(\log T)^2 e^{y/2}}{T}+\frac{1}{y}\right)\, .\end{equation}
The study of the moments of $\Delta(y;t)$ has been a central problem in analytic number theory. It is known (see~\cites{RS94, FJ, Dev}) that, under GRH, $y\mapsto \Delta(y;t)$ is $B^2$ almost-periodic; in particular, its second moment is finite. While Puchta~\cite{Puchta} proved, under GRH, that for fixed $q$ and $k$, the $k$-th moment of $y\mapsto \psi(e^y;\chi)/\exp(y/2)$ is finite, determining the uniform asymptotic behavior of $\frac{1}{Y}\int_{\log 2}^Y (\Delta(u;\chi))^k \mathrm{d}u$ requires additional assumptions, such as LI (see~\cite{Hooley}). More recently, to circumvent the use of LI, de la Bretèche and Fiorilli~\cite{BF} (see also \cite{BFJ}) studied moments of weighted prime counting functions, for which they established uniform lower bounds. The following lemma provides an upper bound for the $2k$-th moment of $\Delta(y;t)$ that is uniform in $k$. For clarity, we postpone its technical proof to the appendix.
\begin{lem}\label{variance}
    Assume GRH for all Dirichlet characters $\chi$ such that $\langle t,\chi\rangle \ne 0$. There exists $\mathcal{C}=\mathcal{C}(t)>1$ such that for all $k\ge 1$ we have \[ \frac{1}{Y}\int_{\log 2}^Y |\Delta(y;t)|^{2k}\mathrm{d}y \le (\mathcal{C} k)^{4k}\qquad (Y\ge 1)\, . \]
\end{lem}
We now put Lemma~\ref{variance} to use: 
\begin{lem}\label{lem:nat-density}
    Assume GRH for all Dirichlet characters $\chi$ such that $\langle t,\chi\rangle \ne 0$. For all $\varepsilon>0$, there exists $Y_0=Y_0(t,\varepsilon)>1$ such that for all $Y\ge Y_0$ we have
    \[ \m \left\{Y\le y<2Y \, :\, |\Delta(y;t)|>(\log y )^{3+\varepsilon}\right\}\le \frac{2}{Y^{\varepsilon/2}}\, , \]
    where $\m$ is the Lebesgue measure on $\R$.
\end{lem}
\begin{proof}
    Let $\varepsilon>0$ and let $Y\ge 2$ be sufficiently large (in terms of $\varepsilon, t,$ and $q$). Let $k=k(Y)$ be the integer such that $2k$ is the least even integer larger than $(\log Y)/(\log \log Y)$. We have, for all $y\ge Y$, $(\log y)^{2k} \ge Y$. Moreover, if $\mathcal{C}=\mathcal{C}(t)$ is a constant given by Lemma~\ref{variance}, then $2k (\log \mathcal{C}k) \sim \log Y$. Since $Y$ is sufficiently large, we have $2k(\log \mathcal{C}k)<(1+\varepsilon/4)\log Y$. Thus $(\mathcal{C}k)^{4k}<Y^{2+\varepsilon/2}$. Applying Markov's inequality, together with Lemma~\ref{variance}, we obtain
    \begin{align*}
        \m \{Y\le y<2Y \, :\, &|\Delta(y;t)|>(\log y )^{3+\varepsilon}\} \le \int_Y^{2Y}\frac{|\Delta(y;t)|^{2k}}{(\log y)^{2k(3+\varepsilon)}}\mathrm{d}y\\
        &\le \frac{1}{Y^{3+\varepsilon}}\int_{\log 2}^{2Y} |\Delta(y;t)|^{2k}\mathrm{d}y
        \le \frac{2Y (\mathcal{C}k)^{4k}}{Y^{3+\varepsilon}}\le\frac{2}{Y^{\varepsilon/2}}\, .
    \end{align*}
    This proves the Lemma.
\end{proof}
Define for all $y\ge1$ \[G(y;t):=\int_{\log 2}^y\Delta(u;t)\mathrm{d}u\, .\]
Our next step is to establish the following auxiliary results, which will be used in the proof of our main results 
\begin{lem}\label{ing}
    If \text{GRH} holds for all Dirichlet characters $\chi$ modulo $q$ satisfying $\langle t,\chi\rangle \ne 0$, then:\begin{enumerate}
        \item Uniformly for $Y\ge 1$: \[\bigl| G(Y;t) \bigr|\ll_t \log Y\,  .\]
        \item There exists $L\in \C$ such that uniformly for $Y \ge1$: \[ \int_{\log 2}^Y\frac{\Delta(y;t)}{y}\mathrm{d}y=L+O_t\left(\frac{\log Y}{Y}\right)\, .\]
        \item The following limit holds: \[\lim_{Y\to \infty} \frac{1}{Y} \int_{2}^Y\frac{|\Delta(u;t)|^2}{\log u}\mathrm{d}u=0\, .\]
    \end{enumerate}
\end{lem}
\begin{proof}
    By~\eqref{expl-form} we have 
    \[G(Y;t)=-\sum_\chi \langle t,\chi \rangle \sum_{0<|\gamma_\chi|\le T}\frac{e^{i\gamma_\chi Y}-e^{i\gamma_\chi \log 2}}{i\gamma_\chi\left(\frac{1}{2}+i\gamma_\chi\right)}+O\left(\frac{(\log T)^2e^Y}{T}+\log Y\right)\, .  \]
    The first bound follows by letting $T$ tend to $\infty$ since the latter sum is absolutely convergent. For the second estimate, we integrate by parts: 
    \[\int_{\log 2}^Y \frac{\Delta(y;t)}{y}\mathrm{d}y=\frac{G(Y;t)}{Y}+\int_{\log 2}^Y\frac{G(y;t)}{y^2}\mathrm{d}y\, .\]
    By the first bound, the latter integral converges as $y\to \infty$; let $L$ denote its limit. We have \[\int_{Y}^\infty \left| \frac{G(y;t)}{y^2}\right|\mathrm{d}y\ll \int_Y^\infty \frac{\log y}{y^2}\mathrm{d}y\ll \frac{\log Y}{Y}\, .\]
    This proves the second estimate. 
    We now turn to the last limit. We write 
    \[ \int_{\log 2}^Y \frac{|\Delta(u;t)|^2}{\log u}\mathrm{d}u=\int_{\log 2}^{\sqrt{Y}}\frac{|\Delta(u;t)|^2}{\log u}\mathrm{d}u+\int_{\sqrt{Y}}^Y\frac{|\Delta(u;t)|^2}{\log u}\mathrm{d}u\, .  \]
    Applying Lemma~\ref{variance} with $k=1$ implies that \begin{align*}
        \int_{\log 2}^{\sqrt{Y}}\frac{|\Delta(u;t)|^2}{\log u}\mathrm{d}u&\ll \int_{\log 2}^{\sqrt{Y}}|\Delta(u;t)|^2\mathrm{d}u\ll \sqrt{Y},\\
        \int_{\sqrt{Y}}^Y\frac{|\Delta(u;t)|^2}{\log u}\mathrm{d}u&\ll\frac{Y}{\log Y}\, ,
    \end{align*}
    which proves that $ \int_{\log 2}^Y |\Delta(u;t)|^2/(\log u)\mathrm{d}u =o(Y)  $. This finishes the proof of the lemma.
\end{proof}
\section{Proofs of the main theorems}
As explained at the beginning of section~\ref{sec:Prel}, we prove our theorems for a general function $t:(\Z/q\Z)^\times\to \C$ satisfying $\langle t,\chi_0\rangle=0$. Our main theorems are obtained by replacing $t$ with $\mathds{1}_{\{a\}}-\mathds{1}_{\{b\}}$. We now state and prove a general version of Theorem~\ref{main}.
\begin{thm}\label{gen-main}
    Assume GRH for all Dirichlet characters modulo $q$ satisfying $\langle t,\chi \rangle\ne0$. There exists a (unique) constant $C=C(t)$ such that for all $\varepsilon>0$, the natural density of the set \[\left\{x\ge 2\, :\, \left| \pi_{1/2}(x;t) +\M(t)\log \log x-C\right|\le \frac{(\log \log x)^{3+\varepsilon}}{\log x}\right\}\]
    exists and equals $1$. Furthermore, there exists $K=K(t)>1$ such that the logarithmic density of the set 
    \[\left\{x\ge 2\, :\, \left| \pi_{1/2}(x;t) +\M(t)\log \log x-C\right|\le \frac{K(\log \log x)}{\log x}\right\}\]
    exists and equals $1$.
\end{thm}
\begin{proof}
    Summation by parts yields, for all $x\ge 2$,
\[\pi_{1/2}(x;t)=\frac{\pi(x;t)}{\sqrt{x}}+\frac{1}{2}\int_2^x\frac{\pi(u;t)}{u^{3/2}}\mathrm{d}u\, .\]
Thus, for all $y\ge 1$
\begin{align*}\pi_{1/2}(e^y;t)&=\frac{\pi(e^y;t)}{\exp(y/2)}+\frac{1}{2}\int_{\log 2}^y\frac{\pi(e^u;t)}{\exp(u/2)}\mathrm{d}u\\
&=\frac{1}{y}\left(\Delta(y;t)-2\M(t) \right)+\frac{1}{2}\int_{\log 2}^y\frac{1}{u}\left(\Delta(u;t)-2\M(t)\right)\mathrm{d}u\, .
\end{align*}
Hence, by part (2) of Lemma~\ref{ing}, denoting $C=\M(t)\log \log 2+L/2$, we have 
\begin{equation}\label{main-form}
    \pi_{1/2}(e^y;t)+\M(t)\log y-C=\frac{1}{y}\left( \Delta(y;t)-2\M(t)+R(y)\right)
\end{equation}
where $|R(y)|\ll \log y$. Let $\varepsilon>0$, and let $Y_0=Y_0(t,\varepsilon)$ be a constant given by Lemma~\ref{lem:nat-density}, chosen large enough so that for all $y\ge Y_0$ we have $|R(y)-2\M(t)|<(\log y)^{3+\varepsilon}/2$. Define  \[E:=\left\{ x\ge  2\, :\, \left|\pi_{1/2}(x;t)+\M(t)\log \log x-C\right|>\frac{(\log \log x)^{3+\varepsilon}}{\log x}\right\}\, .\]
By the triangle inequality and lemma~\ref{lem:nat-density}
\begin{align*}
    \m \left\{\, y\ge \log 2\, :\, e^y\in E\,\right\} &\le Y_0+ \m\left\{  y\ge Y_0\, :\, \left| \Delta(y;t)-2\M(t)+R(y)\right|>(\log y)^{3+\varepsilon} \right\}\\
    &\le Y_0+ \sum_{k\ge 1}\m\left\{ 2^{k-1}Y_0\le y<2^k Y_0\, :\, \bigl| \Delta(y;t) \bigr|>\frac{(\log y)^{3+\varepsilon}}{2}\right\}\\
    &\le Y_0+ \sum_{k\ge1}\frac{1}{(2^{k-1}Y_0)^{\varepsilon/2}}\ll1\, .
\end{align*}
This establishes that \[ \int_{\log 2}^\infty \mathds{1}_E(e^y)\mathrm{d}y <\infty \, . \]
Thus, by~\eqref{nat-density}, the natural density of the set $E$ exists and equals $0$. Hence, the natural density of its complement exists and equals $1$. We now move to the statement regarding the logarithmic density. Let $K$ be sufficiently large so that $|R(y)-2\M(t)|<(K-1)\log y$. Define 
\[E':=\left\{ x\ge  2\, :\, \left|\pi_{1/2}(x;t)+\M(t)\log \log x-C\right|>\frac{K(\log \log x)}{\log x}\right\}\, .\]
By the triangle inequality, together with Markov's inequality we have for all $Y\ge 2$
\begin{align*}
    \m \left\{ \log 2\le y\le Y\, :\, e^y\in E'\right\} &\le \m\left\{  \log 2\le y\le Y\, :\, \left| \Delta(y;t)-2\M(t)+R(y)\right|>K\log y  \right\}\\
    &\le \m\left\{  \log 2\le y\le Y\, :\, \left| \Delta(y;t)\right|> \log y \right\}\\
    &\le 3+ \int_{3}^Y \frac{|\Delta(y;t)|^2}{(\log y)^2}\mathrm{d}y \le 3+ \int_{3}^Y \frac{|\Delta(y;t)|^2}{\log y}\mathrm{d}y \, .
\end{align*}
Using part (3) of Lemma~\ref{ing}, we deduce that the logarithmic density of the set $E'$ exists and equals $0$, hence the logarithmic density of its complement exists and equals $1$.
\end{proof}
We now move to Theorem~\ref{GRHalmimpliesDRH}
\begin{proof}[Proof of Theorem~\ref{GRHalmimpliesDRH}]
    For $x\ge 2$, define $F(x)=(\log x)^{m_\chi} \prod_{p\le x} \bigl(\,1-\chi(p)/\sqrt{p}\,\bigr)$. Following the proof of~\cite{Ao-Ko}*{Proposition 2.1}, and using the following form of Mertens' theorem \[ \sum_{p\le x} \frac{\chi(p^2)}{p}=\langle \chi,r\rangle \log \log x+c+O\left(\frac{1}{\log x}\right) \, ,\]
    where $c\in \C$ is a constant, we deduce that there exists a constant $C'\in \C$ such that 
    \[  \log F(x)= \pi_{1/2}(x,\chi)+\M(\chi)\log \log x+C'+O\left(\frac{1}{\log x}\right)\, .\]
    Set $t=\chi$, and let $\varepsilon>0$ and $C=C(\chi)$ be the constant given by~\eqref{main-form}. Define $\ell_\chi =\exp(C+C')\ne 0$, and fix $x\ge 2$ sufficiently large satisfying the inequality
    \[ \left|\pi_{1/2}(x,\chi)+\M(\chi)\log \log x-C\right|<\frac{(\log \log x)^{3+\varepsilon/2}}{\log x}\, . \]
    Using the inequality $|\exp(z)-1|\le 2|z|$ for $|z|<1/2$, we deduce that \begin{align*} |F(x)-\ell_\chi|&=|\ell_\chi|\,\bigl|\exp(\log F(x)-C-C')-1\bigr|\\
    &\le 2|\ell_\chi|\, \left|\pi_{1/2}(x,\chi)+\M(\chi)\log \log x-C+O\left(\frac{1}{\log x}\right)\right|\\
    &< \frac{(\log \log x)^{3+\varepsilon}}{\log x}\, ,
    \end{align*}
    since $x$ is sufficiently large. Since $\mathcal{E}_\chi$ contains all sufficiently large values of a set of natural density $1$, it follows that the natural density of $\mathcal{E}_\chi$ exists and equals $1$. 
\end{proof}
We now prove Theorem~\ref{mean} for a general function $t\, :\, (\Z/q\Z)^\times\to \C$.
\begin{proof}[Proof of Theorem~\ref{mean}]
    Let $x\ge 2$ and set $Y=\log x$.
    From~\eqref{main-form} we deduce that 
    \[\int_{\log 2}^Y e^y\pi_{1/2}(e^y;t)\mathrm{d}y=\int_{\log 2}^Y\left(-\M(t)e^y \log y+Ce^y\right)\mathrm{d}y+\int_{\log 2}^Y\frac{e^y}{y}\left(\Delta(y;t)-2\M(t)+R(y)\right)\mathrm{d}y\, .\]
    Applying part (1) of Lemma~\ref{ing} together with the bound $ \int_{\log 2}^Y\tfrac{e^y}{y}\mathrm{d}y\ll e^Y/Y$, we deduce that \[ \left| \int_{\log 2}^Y \frac{e^y \Delta(y;t)}{y}\mathrm{d}y\right|=\left| \frac{e^Y}{Y}G(Y;t)-\int_{\log 2}^Y G(y;t)\frac{ye^y-e^y}{y^2}\mathrm{d}y\right|\ll \log Y \frac{e^Y}{Y}\, .\]
    Since $|R(y)|\ll \log y$, we have
    \[ \left|\int_{\log 2}^Y \frac{e^y}{y}\left(\Delta(y;t)-2\M(t)+R(y)\right)\mathrm{d}y\right|\ll \log Y\frac{e^Y}{Y}\, .\]
    Using the classical estimate 
    \[ \int_{\log 2}^Y e^y \log y\,  \mathrm{d}y=e^Y \log Y+O\left(\frac{e^Y}{Y}\right)\, , \]
    we conclude that
    \[\int_2^x\pi_{1/2}(u;t)\mathrm{d}u =\int_{\log 2}^Ye^y\pi_{1/2}(e^y;t) \mathrm{d}y=-\M(t)e^Y\log Y+Ce^Y+O\left(\frac{e^Y \log Y}{Y}\right)\, .\]
    The result follows by substituting $Y=\log x$.
\end{proof}

\section*{Appendix: Proof of Lemma~\ref{variance}}
Our proof of Lemma~\ref{variance} is inspired by Puchta's method for establishing the finiteness of the $k$-th moments~\cite{Puchta}*{Theorem 1}. While Puchta focused on the case of fixed $k$, a naive quantification of his arguments would yield a growth rate of $(\mathcal{C}k)^{8k}$. By refining several steps in the estimates of the multisums over zeros, we are able to sharpen this bound to $(\mathcal{C}k)^{4k}$.
\begin{proof}[Proof of Lemma~\ref{variance}]
    We first write the error term in~\eqref{expl-form} as $\mathcal{R}_T(y;t)\le C_t\left(\frac{(\log T)^2 e^{y/2}}{T}+\frac{1}{y}\right)$. We have 
    \begin{equation}\label{power-inequ} |\Delta(y;t)|^{2k} \le 2^{2k}\left( \left|\sum_{\chi \ne\chi_0}\langle t,\chi\rangle\sum_{0<|\gamma_\chi|\le T} \frac{e^{iy\gamma_\chi}}{\frac{1}{2}+i\gamma_\chi}\right|^{2k}+\left| \mathcal{R}_T(y;t)\right|^{2k} \right)\, .\end{equation}
    Let $\rho_1,\rho_2,\dots$ denote the zeros of all Dirichlet $L$-functions $L(s,\chi)$, with $\chi \ne \chi_0$, chosen so that their respective imaginary parts $\gamma_1,\gamma_2,\dots$ satisfy $|\gamma_1|\le |\gamma_2|\le \dots$ This enables us to write:
    \[\sum_{\chi \ne\chi_0}\langle t,\chi\rangle\sum_{0<|\gamma_\chi|\le T} \frac{e^{iy\gamma_\chi}}{\frac{1}{2}+i\gamma_\chi}=\sum_{\substack{n\ge1\\ |\gamma_n|\le T}} a_n \frac{e^{i y \gamma_n}}{\rho_n}\, .\]
    Integrating in~\eqref{power-inequ}, using the triangle inequality, then letting $T\to \infty$, yields
    {\allowdisplaybreaks \begin{align*}
        \int_{\log 2}^Y|&\Delta(y;t)|^{2k}\mathrm{d}y \le 4^k \left(\sum_{\substack{n_1,\dots,n_{2k}\\ |\gamma_{n_j}|\le T}} \prod_{i=1}^{2k}\frac{|a_{n_i}|}{|\rho_{n_i}|}\int_{\log 2}^Y e^{iy(\gamma_{n_1}+\dots+\gamma_{n_k}-\gamma_{n_{k+1}}-\dots-\gamma_{n_{2k}})}\mathrm{d}y+ \int_{\log 2}^Y \left|\mathcal{R}_T(y;t)\right|^{2k}\mathrm{d}y \right)\\
        &\le (2\max_\chi |\langle t,\chi\rangle|)^{2k} \sum_{n_1,\dots,n_{2k}}\left(\prod_{i=1}^{2k}\frac{1}{|\rho_{n_i}|}\right)\min\left(Y;\tfrac{1}{|\gamma_{n_1}+\dots+\gamma_{n_k}-\gamma_{n_{k+1}}\dots-\gamma_{n_{2k}}|}\right)+O\left((2C_t)^{2k}\right)\\ 
        &= (2\max_\chi |\langle t,\chi\rangle|)^{2k} \sum_{n_1,\dots,n_{2k}}\left(\prod_{i=1}^{2k}\frac{1}{|\rho_{n_i}|}\right)\min\left(Y;\frac{1}{|\gamma_{n_1}+\dots+\gamma_{n_{2k}}|}\right)+O\left((2C_t)^{2k}\right)\, ,
    \end{align*}}% 
    \noindent where the last equality is given by the symmetry of zeros, since $\gamma$ is the imaginary part of a zero of $L(s,\chi)$ if and only if $-\gamma$ is the imaginary part of a zero of $L\bigl(s,\overline{\chi}\bigr )$ with the same multiplicity. Define $f(n_1,\dots,n_{2k})=\bigl(\prod_{i=1}^{2k}1/|\rho_{n_i}|\bigr)\min\bigl(Y;1/|\gamma_{n_1}+\dots+\gamma_{n_{2k}}|\bigr)$. Since $f$ is a symmetric function, we have 
    \[\sum_{n_1,\dots,n_{2k}\ge 1} f\left(n_1,\dots,n_{2k}\right)\le 2\binom{2k}{2} \sum_{n_1,\dots,n_{2k-2}\ge 1}\   \sum_{\substack{n_{2k}\ge n_{2k-1}\ge n_j\\ \text{for all } j\le 2k-2}} f(n_1,\dots,n_{2k})\, .\]
    To prove the lemma, it suffices to prove that
    \[\sum_{n_1,\dots,n_{2k-2}\ge 1}\   \sum_{\substack{n_{2k}\ge n_{2k-1}\ge n_j\\ \text{for all } j\le 2k-2}} f(n_1,\dots,n_{2k}) \ll Y(\mathcal{B}k)^{4k-2}\, ,\]
    for some large constant $\mathcal{B}$ depending only on $q$.
    We decompose the latter sum into $S_1+S_2$, where $S_1$ is the sum over $2k$-tuples such that $\bigl| \gamma_{n_1}+\dots+\gamma_{n_k}|\ge \bigl|\rho_{n_{2k}}\bigr|^{1/2}$, and $S_2$ is the sum over those satisfying $|\gamma_{n_1}+\dots+\gamma_{n_{2k}}|<\left|\rho_{n_{2k}}\right|^{1/2}$.
    Combining the standard upper bound $\sum_{|\rho|\le T} 1/|\rho| \le c_q (\log T)^2$, for some $c_q\ge 1$, with the inequality $(\log t)^n \le n^n t^{1/e}$ (which is obtained by studying the function $t\mapsto (\log t)^n t^{-1/e}$), we deduce that
    \begin{align*} S_1 &\le \sum_{n_1,\dots,n_{2k-2}\ge 1}\   \sum_{\substack{n_{2k}\ge n_{2k-1}\ge n_j\\ \text{for all } j\le 2k-2}} \frac{1}{|\rho_{n_1}\dots \rho_{n_{2k-1}}|\, |\rho_{n_{2k}}|^{3/2}}\\
    &\le c_q^{2k-1} \sum_{\rho}\frac{|\log (|\rho|+1)|^{4k-2}}{|\rho|^{3/2}}\\
    &\ll (c_q(4k-2))^{4k-2} \sum_{\rho}\frac{1}{|\rho|^{3/2 - 1/e}}\ll (\mathcal{B}k)^{4k-2}\, ,
    \end{align*}
    for some $\mathcal{B}\ge 4c_q$. For fixed $\gamma_{n_1},\dots,\gamma_{n_{2k-1}}$ and $s=\bigl| \gamma_{n_1}+\dots+\gamma_{n_{2k-1}}\bigr|$, the number of $\gamma_{n_{2k}}$ satisfying $\bigl| \gamma_{n_1}+\dots+\gamma_{n_{2k}}\bigr|<\bigl|\rho_{n_{2k}}\bigr|^{1/2}$ is $\ll \sqrt{s}\log (s+2)$. Since $|s| \le 2k|\rho_{n_{2k-1}}|$, we have 
    \begin{align*}
        S_2 &\ll Y\sqrt{2k}\sum_{n_1,\dots,n_{2k-2}\ge 1}\  \sum_{\substack{ n_{2k-1}\ge n_j\\ \text{for all } j\le 2k-2}} \frac{\sqrt{|\rho_{n_{2k-1}}|}\log(2k(|\rho_{n_{2k-1}}|+1))}{|\rho_{n_1}\dots\rho_{n_{2k-1}}|\, |\rho_{n_{2k-1}}|}\\
        &\ll Y c_q^{2k-2} (2k) \sum_{\rho}\frac{\log(|\rho|+1)^{4k-3}}{|\rho|^{3/2}}\\
        &\ll Y (4c_qk)^{4k-2} \sum_\rho \frac{1}{|\rho|^{3/2-1/e}}\ll Y (\mathcal{B}k)^{4k-2} \, .
    \end{align*}
    This finishes the proof of the lemma.
    
\end{proof}


\begin{bibdiv}
\begin{biblist}

\bib{Ao-Ko}{article}{
   author={Aoki, Miho},
   author={Koyama, Shin-ya},
   title={Chebyshev's bias against splitting and principal primes in global
   fields},
   journal={J. Number Theory},
   volume={245},
   date={2023},
   pages={233--262},
   issn={0022-314X},
   review={\MR{4517481}},
   doi={10.1016/j.jnt.2022.10.005},
}

\bib{BF}{article}{
   author={de la Bret\`eche, R.},
   author={Fiorilli, D.},
   title={Moments of moments of primes in arithmetic progressions},
   journal={Proc. Lond. Math. Soc. (3)},
   volume={127},
   date={2023},
   number={1},
   pages={165--220},
   issn={0024-6115},
   review={\MR{4611407}},
   doi={10.1112/plms.12542},
}

\bib{BFJ}{article}{
   author={de la Bret\`eche, R\'egis},
   author={Fiorilli, Daniel},
   author={Jouve, Florent},
   title={Moments in the Chebotarev density theorem: general class
   functions},
   journal={Algebra Number Theory},
   volume={19},
   date={2025},
   number={3},
   pages={481--520},
   issn={1937-0652},
   review={\MR{4879359}},
   doi={10.2140/ant.2025.19.481},
}

\bib{Chebyshev}{article}{
   author={Chebyshev, Pafnouti},
   title={Lettre de M. le Professeur Tchébychev à M. Fuss sur un nouveau théorème relatif aux nombres premiers contenus dans les formes 4n + 1 et 4n + 3},
   journal={Bull. Classe Phys. Acad. Imp. Sci. St. Petersburg},
   pages={208},
   year={1853},
}

\bib{Conrad}{article}{
   author={Conrad, Keith},
   title={Partial Euler products on the critical line},
   journal={Canad. J. Math.},
   volume={57},
   date={2005},
   number={2},
   pages={267--297},
   issn={0008-414X},
   review={\MR{2124918}},
   doi={10.4153/CJM-2005-012-6},
}

\bib{Dev}{article}{
   author={Devin, Lucile},
   title={Chebyshev's bias for analytic $L$-functions},
   journal={Math. Proc. Cambridge Philos. Soc.},
   volume={169},
   date={2020},
   number={1},
   pages={103--140},
}

%\bib{Fiorilli}{article}{
 %  author={Fiorilli, Daniel},
  % title={Highly biased prime number races},
  % journal={Algebra Number Theory},
  % volume={8},
  % date={2014},
  % number={7},
  % pages={1733--1767},
  % issn={1937-0652},
  % review={\MR{3272280}},
  % doi={10.2140/ant.2014.8.1733},
%}

\bib{FJ}{article}{
   author={Fiorilli, Daniel},
   author={Jouve, Florent},
   title={Distribution of Frobenius elements in families of Galois
   extensions},
   journal={J. Inst. Math. Jussieu},
   volume={23},
   date={2024},
   number={3},
   pages={1169--1258},
}

\bib{Gallagher}{article}{
   author={Gallagher, P. X.},
   title={Some consequences of the Riemann hypothesis},
   journal={Acta Arith.},
   volume={37},
   date={1980},
   pages={339--343},
   issn={0065-1036},
   review={\MR{0598886}},
   doi={10.4064/aa-37-1-339-343},
}

\bib{Hooley}{article}{
   author={Hooley, C.},
   title={On the Barban-Davenport-Halberstam theorem. XVII},
   conference={
      title={Proceedings of the Session in Analytic Number Theory and
      Diophantine Equations},
   },
   book={
      series={Bonner Math. Schriften},
      volume={360},
      publisher={Univ. Bonn, Bonn},
   },
   date={2003},
   pages={5},
   review={\MR{2075630}},
}

\bib{Kac}{article}{
   author={Kaczorowski, Jerzy},
   title={On the distribution of primes (mod $4$)},
   journal={Analysis},
   volume={15},
   date={1995},
   number={2},
   pages={159--171},
   %issn={0174-4747},
   %review={\MR{1344249}},
   %doi={10.1524/anly.1995.15.2.159},
}


\bib{Ko-Ku}{article}{
   author={Koyama, Shin-ya},
   author={Kurokawa, Nobushige},
   title={Chebyshev's bias for Ramanujan's $\tau$-function via the deep
   Riemann hypothesis},
   journal={Proc. Japan Acad. Ser. A Math. Sci.},
   volume={98},
   date={2022},
   number={6},
   pages={35--39},
   issn={0386-2194},
   review={\MR{4432981}},
   doi={10.3792/pjaa.98.007},
}

\bib{Littlewood}{article}{
  author={Littlewood, J. E.},
  title={Sur la distribution des nombres premiers},
  journal={Comptes Rendus},
  volume={158},
  date={1914},
  pages={1869--1872},
}

\bib{Montg}{article}{
   author={Montgomery, H. L.},
   title={The zeta function and prime numbers},
   conference={
      title={Proceedings of the Queen's Number Theory Conference, 1979},
      address={Kingston, Ont.},
      date={1979},
   },
   book={
      series={Queen's Papers in Pure and Appl. Math.},
      volume={54},
      publisher={Queen's Univ., Kingston, ON},
   },
   date={1980},
   pages={1--31},
   review={\MR{0634679}},
}

\bib{MV}{book}{
   author={Montgomery, Hugh L.},
   author={Vaughan, Robert C.},
   title={Multiplicative number theory. I. Classical theory},
   series={Cambridge Studies in Advanced Mathematics},
   volume={97},
   publisher={Cambridge University Press, Cambridge},
   date={2007},
   pages={xviii+552},
   isbn={978-0-521-84903-6},
   review={\MR{2378655}},
}

\bib{Puchta}{article}{
   author={Puchta, J.-C.},
   title={On large oscillations of the remainder of the prime number
   theorems},
   journal={Acta Math. Hungar.},
   volume={87},
   date={2000},
   number={3},
   pages={213--227},
   issn={0236-5294},
   review={\MR{1761276}},
   doi={10.1023/A:1006711604010},
}

\bib{RS94}{article}{
   author={Rubinstein, Michael},
   author={Sarnak, Peter},
   title={Chebyshev's bias},
   journal={Experiment. Math.},
   volume={3},
   date={1994},
   number={3},
   pages={173--197},
}

\bib{Sheth}{article}{
   author={Sheth, Arshay},
   title={Euler products at the centre and applications to Chebyshev's bias},
   journal={Math. Proc. Cambridge Philos. Soc.},
   volume={179},
   date={2025},
   number={2},
   pages={331--349},
   issn={0305-0041},
   review={\MR{4945971}},
   doi={10.1017/S0305004125000234},
}

\bib{Su}{article}{
   author={Suzuki, Masatoshi},
   title={On variants of Chebyshev's conjecture},
   journal={Ramanujan J.},
   volume={68},
   date={2025},
   number={4},
   pages={95},
   issn={1382-4090},
   review={\MR{4985504}},
   doi={10.1007/s11139-025-01238-9},
}
    
\end{biblist}

\end{bibdiv}
\end{document}